\documentclass[11pt,reqno]{amsart}
\usepackage{amscd,amsmath,amsopn,amssymb,amsthm,multicol}
\usepackage[breaklinks=true,colorlinks=true,linkcolor=blue,citecolor=blue,urlcolor=blue]{hyperref}

\DeclareMathOperator{\ad}{ad}

\DeclareMathOperator{\diag}{diag}

\DeclareMathOperator{\Ad}{Ad}

\DeclareMathOperator{\Lin}{Lin}

\DeclareMathOperator{\Isom}{Isom}
\DeclareMathOperator{\trace}{trace}

\DeclareMathOperator{\Ker}{Ker}
\DeclareMathOperator{\Lie}{Lie}

\renewenvironment{proof}[1][Proof]{\textbf{#1.} }
{\ \rule{0.5em}{0.5em}}
\newtheorem{theorem}{Theorem}
\newtheorem{prop}{Proposition}
\newtheorem{lemma}{Lemma}
\newtheorem{corollary}{Corollary}
\newtheorem{quest}{Question}

\theoremstyle{definition}

\newtheorem{remark}{Remark}

\begin{document}

\title
[On left-invariant Einstein Riemannian metrics \dots] {On left-invariant Einstein Riemannian metrics \\ that are not geodesic orbit}

\author{Yu.G.~Nikonorov}

\begin{abstract}
In this paper we prove
that the compact Lie group $G_2$ admits a left-invariant Einstein metric that is not geodesic orbit.
In order to prove the required assertion, we develop some special tools for geodesic orbit Riemannian manifolds.
It should be noted that a suitable metric is discovered in a recent paper by I.~Chrysikos and Y.~Sakane,
where the authors proved also that this metric is not naturally reductive.

\vspace{2mm} \noindent 2010 Mathematical Subject Classification:
53C20, 53C25, 53C35.

\vspace{2mm} \noindent Key words and phrases: homogeneous spaces, homogeneous Riemannian manifolds, geodesic orbit
spaces, Einstein manifolds.
\end{abstract}

\maketitle

\section{Introduction}

All manifolds in this paper are supposed to be connected.
Our study related to the following important classes of Riemanian manifolds: geodesic orbit manifolds and homogeneous Einstein manifolds.
Recall that a Riemannian manifold $(M=G/H,g)$, where $H$ is a compact subgroup
of a Lie group $G$ and $g$ is a $G$-invariant Riemannian metric,
is called a {\it space with homogeneous geodesics} or {\it geodesic orbit space} if any geodesic $\gamma $ of $M$ is an orbit of
a 1-parameter subgroup of the group $G$. Moreover,
a Riemannian manifold $(M,g)$ is called a  {\it manifold with
homogeneous geodesics} or {\it geodesic orbit manifold} if any
geodesic $\gamma $ of $M$ is an orbit of a 1-parameter subgroup of
the full isometry group of $(M,g)$.
Hence, a Riemannian manifold $(M,g)$ is  a geodesic orbit Riemannian manifold,
if it is a geodesic orbit space with respect to its full connected isometry group.
This terminology was introduced in
\cite{KV} by O.~Kowalski and L.~Vanhecke, who initiated a systematic study of such spaces.

Let $(M=G/H, g)$ be a homogeneous Riemannian manifold. Since $H$
is compact, there is an $\Ad(H)$-invariant decomposition
\begin{equation}\label{reductivedecomposition}
\mathfrak{g}=\mathfrak{h}\oplus \mathfrak{m},
\end{equation}
where $\mathfrak{g}={\rm Lie }(G)$ and $\mathfrak{h}={\rm Lie}(H)$.
The Riemannian metric $g$ is $G$-invariant and is determined
by an $\Ad(H)$-invariant Euclidean metric $g = (\cdot,\cdot)$ on
the space $\mathfrak{m}$ which is identified with the tangent
space $T_oM$ at the initial point $o = eH$.

By $[\cdot, \cdot]$ we denote the Lie bracket in $\mathfrak{g}$, and by
$[\cdot, \cdot]_{\mathfrak{m}}$ its $\mathfrak{m}$-component according to (\ref{reductivedecomposition})
We recall (in the above terms) a well-known criteria.

\begin{lemma}[\cite{KV}]\label{GO-criterion}
A homogeneous Riemannian manifold   $(M=G/H,g)$ with the reductive
decomposition  {\rm (\ref{reductivedecomposition})} is a geodesic orbit space if and
only if  for any $X \in \mathfrak{m}$ there is $Z \in \mathfrak{h}$ such that
$([X+Z,Y]_{\mathfrak{m}},X) =0$ for all $Y\in \mathfrak{m}$.
\end{lemma}

In what follows, the latter condition in this lemma will be called the {\it GO-property}. Recall that for a given $X\in\mathfrak{m}$, this means that the orbit
of  $\exp\bigl((X+Z)t\bigr) \subset G$, $t \in \mathbb{R}$, through the point $o=eH$ is a geodesic in $(M=G/H,g)$.

Various useful information on geodesic orbit manifolds could be found in the papers \cite{KV,Gor96,NRS,Nik2016} and in the references therein.
There are some important subclasses of geodesic orbit manifolds, for instance, symmetric spaces \cite{Hel}, isotropy irreducible spaces~\cite{Bes},
weakly symmetric spaces \cite{W1}.

Any homogeneous space
$M = G/H$ of a compact Lie group $G$  admits a Riemannian metric $g$ such that $(M,g)$ is a geodesic orbit space.
Indeed, it suffices to take the metric
$g$ induced by a bi-invariant Riemannian metric $g_0$ on the Lie group  $G$ such that
$ (G,g_0) \to (M=G/H, g)$ is a Riemannian submersion
with totally geodesic fibres. Such geodesic orbit space $(M = G/H, g)$ is called  a {\it normal homogeneous space}.
If in addition $g_0$ is generated with the minus Killing form of the Lie algebra $\Lie(G)$, then the metric $g$ is called {\it standard} or {\it Killing}.
It should be noted also that  any  naturally  reductive  Riemannian manifold is geodesic orbit.
Recall that a Riemannian manifold $(M,g)$ is
{\it naturally reductive} if it admits a transitive Lie group $G$ of isometries with a bi-invariant
pseudo-Riemannian metric $g_0$, which induces the metric $g$ on $M = G/H$ (see  \cite{Bes} and \cite{KN}).
Clear that symmetric spaces and normal homogeneous spaces are naturally reductive.
The classification of  naturally reductive homogeneous spaces of $\dim \leq 5$ was
obtained by O. Kowalski and L.~Vanhecke in~\cite{KV1}.
In \cite{KV}, O.~Kowalski and L.~Vanhecke classified all geodesic orbit spaces
of dimension $\leq 6$. In particular, they proved that every geodesic orbit Riemannian manifold of dimension $\leq 5$ is naturally reductive.

Finally, we notice  that
{\it generalized normal homogeneous Riemannian manifolds}
({\it $\delta$-homogeneous manifold}, in another terminology)  and
{\it Clifford--Wolf homogeneous Riemannian manifolds} constitute other important
subclasses of geodesic orbit manifolds \cite{BerNik, BerNikClif}.
For more details on geodesic orbit Riemannian manifolds see the papers \cite{KV,Gor96,Nik2016,NRS} and the references therein.
\medskip

A Riemannian metric is {\it Einstein} if the Ricci curvature is a constant multiple of the metric.
Various results on Einstein manifolds could be found in the book \cite{Bes} of A.L.~Besse and in more recent surveys \cite{NRS, Wang1, Wang2}.
Well known examples of Einstein manifolds are irreducible symmetric spaces and isotropy irreducible spaces. It should be noted that
there are many homogeneous examples of Einstein metrics, in particular, normal homogeneous and naturally reductive.
\smallskip

Now, we are going to study the following problem: {\it To find an example of a compact
Lie group $G$ supplied with a left-invariant Riemannian $\rho$ such that
$(G,\rho)$ is Einstein but is not a geodesic orbit Riemannian manifold}. If we replace the GO-property with a more weak condition of the natural reductivity,
then we get a well known problem. Let us recall some details.
\smallskip

In \cite{DZ}, J.E.~D'Atri and W.~Ziller have investigated naturally
reductive metrics among the left invariant metrics on compact Lie
groups, and have given a complete classification in the case of
simple Lie groups. Let $G$ be a compact connected semi-simple Lie
group, $H$ a closed subgroup of $G$, and let ${\mathfrak g}$ be the
Lie algebra of $G$ and ${\mathfrak h}$ the subalgebra corresponding
to $H$. Denote by $\langle \cdot, \cdot \rangle$ the negative of the Killing form $B$ of
${\mathfrak g}$. Let ${\mathfrak m}$ be a orthogonal complement of
${\mathfrak h}$ with respect to $B$. Then we have
$$
{\mathfrak g}={\mathfrak h}\oplus {\mathfrak m},\quad \Ad(H){\mathfrak m}\subset{\mathfrak m}.
$$
Let ${\mathfrak h}={\mathfrak h}_0\oplus {\mathfrak h}_1\oplus\cdots\oplus{\mathfrak h}_p$ be the decomposition into
ideals of ${\mathfrak h}$, where ${\mathfrak h}_0$ is the center of
${\mathfrak h}$ and ${\mathfrak h}_i$, $i=1,\cdots,p$, are simple
ideals of ${\mathfrak h}$. Let $A_0|_{{\mathfrak h}_0}$ be an
arbitrary inner product on~${\mathfrak h}_0$.

\begin{prop}[\cite{DZ}]\label{natredgr}
Under the notations above, a left
invariant metric on $G$ of the form
\begin{equation}\label{natural}
(\cdot,\cdot)=x \langle \cdot, \cdot \rangle |_{\mathfrak m}+A_0|_{\mathfrak h_0}+u_1 \langle \cdot, \cdot \rangle|_{\mathfrak h_1}+
\cdots+u_p \langle \cdot, \cdot \rangle|_{\mathfrak h_p}
\qquad (\,x, u_1, \dots, u_p \in {\mathbb R}^+\,)
\end{equation}
is
naturally reductive with respect to $G\times H$, where $G\times H$
acts on $G$ by $(a,b) (c) = acb^{-1}$. Conversely, if a left invariant
metric $(\cdot,\cdot)$ on a compact simple Lie group $G$ is
naturally reductive, then there exists a closed subgroup $H$ of $G$
such that the metric $(\cdot,\cdot)$ is given by the form
{\rm(\ref{natural})}.
\end{prop}
\smallskip

In the paper \cite{DZ}, J.E. D'Atri and W. Ziller  found a large number of left-invariant Einstein metrics, which are naturally reductive,
on the compact Lie groups $SU(n)$, $SO(n)$ and $Sp(n)$.
In the same paper, the authors posed the problem of existence
of left-invariant Einstein metrics, which are not naturally reductive,
on compact Lie groups.
\smallskip

The first examples were obtained by K. Mori in \cite{M}, where he constructed non naturally reductive Einstein metrics on the
Lie group $SU(n)$ for $n\ge 6$. Further, in \cite{AMS} A.~Arvanitoyeorgos, Y.~Sakane, and K. Mori
proved existence of non naturally reductive Einstein metrics on the compact Lie groups
$SO(n)$ $(n\ge 11)$, $Sp(n)$ $(n\ge 3)$, $E_6, E_7$ and $E_8$.
In \cite{CL} Z. Chen and K. Liang found three naturally reductive and one non naturally reductive Einstein metric on the compact Lie group $F_4$.
Also, in \cite{ASS2}, the authors  obtained new left-invariant Einstein metrics on the symplectic group $Sp(n)\ (n\ge 3)$, and in
\cite{CS} I.~Chrysikos and Y.~Sakane obtained new non naturally reductive Einstein metrics on exceptional Lie groups.
In the recent paper \cite{ASS3}, the authors obtained new left-invariant Einstein metrics on the compact
Lie groups $SO(n)$ $(n\ge 7)$ which are not naturally reductive.
\smallskip

Note that there are examples of homogeneous Einstein metrics (distinct from compact Lie groups with left-invariant Riemannian metrics),
that are neither natural reductive, nor geodesic orbit. Let us discuss
some important results in this direction.

It is well known, that all homogeneous Einstein Riemannian manifolds of dimension $\leq 4$ are locally symmetric, hence, naturally reductive and
geodesic orbit, see e.~g. \cite{Bes}.

There is a 5-dimensional non-compact homogeneous Einstein Riemannian manifold, that is not geodesic orbit. Indeed,
C.~Gordon proved that
{\it every Riemannian geodesic orbit manifold of nonpositive Ricci curvature is symmetric} \cite{Gor96} (see also \cite{Nik2013n}
for a more short proof of this result).
On the other hand, there are $5$-dimensional non-symmetric Einstein homogeneous manifold with negative Ricci curvature.
Moreover, there is a one-parameter family of pairwise non-isometric 5-dimensional (non-symmetric) Einstein solvmanifolds, see \cite{Al} or \cite{Nik2005}.

There are 5-dimensional compact homogeneous Einstein Riemannian manifolds, that are not geodesic orbit.
These are the spaces $M_{a,b}=SU(2)\times SU(2)/S^1$, where the circle $S^1=S^1(a,b)$ is given by
$e^{2\pi {\bf i} t} \mapsto
\left(
\begin{array}{cc}
e^{2 a\pi {\bf i} t}&0\\
0&e^{-2 a\pi {\bf i} t}\\
\end{array}
\right)
\times
\left(
\begin{array}{cc}
e^{2 b \pi {\bf i} t}&0\\
0&e^{-2 b\pi {\bf i} t}\\
\end{array}
\right),
$
$t \in \mathbb{R}$, and $(a, b)$ is a nontrivial vector with
integer coordinates such that $a\geq b \geq 0$ and $\operatorname{gcd} (a,b)= 1$.
For every pair $(a,b)\neq (1,1)$, the space $M_{a,b}$ admits exactly one invariant Einstein metric $\rho_{a,b}$ up to a homothety, see
\cite{Rod} and \cite{WZ4}. As it was remarked in \cite{Rod}, all these Einstein metrics  are not naturally reductive, that follows from the results of~\cite{KV1}.
Note also that the results of~\cite{KV} imply that all Riemannian manifolds  $(M_{a,b}, \rho_{a,b})$ are not geodesic orbit.

There is a 6-dimensional compact homogeneous Einstein Riemannian manifold, that is not geodesic orbit.
This is a flag manifold $SU(3)/T_{\max}$ with a K\"{a}hler--Einstein invariant metric (it is different from the $SU(3)$-normal one), see details e.~g. in
\cite{NikRod}. The fact that it is not geodesic orbit follows also from the results of \cite{KV}.

There are 7-dimensional compact homogeneous Einstein Riemannian manifolds, that are not geodesic orbit. For instance, we can consider
the Aloff--Wallach spaces $W_{k,l}$, $(k,l)\neq (1,0)$ and  $(k,l)\neq (1,1)$, supplied with any invariant Einstein metrics.
Each such spaces admits exactly two invariant Einstein metrics, see e.~g. \cite{Nik2004}. Moreover,
using the Einstein equations (1) in \cite{Nik2004}, it is easy to check that
we never have the equality $x_1=x_2=x_3$ for these Einstein metrics. Hence, all these metrics are not geodesic orbit according to the results of
Section \ref{AWspaces}.
\smallskip

The paper is organized as follows: In Section 2 we discuss the problem of description of geodesic orbit Rimannian metrics in terms
of non full isometry groups. In some special cases such reductions could be very useful, see e.~g. Theorem \ref{nonfull3}
for the case of a transitive normal subgroup in the full connected isometry group of a compact geodesic orbit manifold.
In Section 3, we apply Theorem \ref{nonfull3} to describe geodesic orbit Riemannian metrics on the Aloff--Wallach spaces.
Section 4 is devoted to the description of left-invariant geodesic orbit Riemannian metrics on compact simple Lie groups.
On the base of the obtained results,  in  Section 5 we prove  Theorem \ref{einstnongo} stated that
the compact simple Lie group $G_2$ admits a left-invariant Riemannian  metric $\rho$
such that the Riemannian manifold $(G_2,\rho)$ is Einstein but
is not geodesic orbit. We also discuss some related unsolved questions in this last section.

\section{On reductions to smaller isometry groups}

Let $(M=G_1/H_1, g)$ be a homogeneous Riemannian space with compact $H_1$
is and an $\Ad(H_1)$-invariant decomposition
\begin{equation}\label{reductivedecomposition1}
\mathfrak{g}_1=\mathfrak{h}_1\oplus \mathfrak{m}_1,
\end{equation}
where $\mathfrak{g}_1={\rm Lie }(G_1)$ and $\mathfrak{h}_1={\rm Lie}(H_1)$.
The Riemannian metric $g$ is $G_1$-invariant and is determined
by an $\Ad(H_1)$-invariant Euclidean metric $g = (\cdot,\cdot)_1$ on
the space $\mathfrak{m}_1$ which is identified with the tangent
space $T_oM$ at the initial point $o = eH_1$.

We denote the Lie bracket in $\mathfrak{g}_1$ by $[\cdot, \cdot]$, and its $\mathfrak{m}_1$-component according to~(\ref{reductivedecomposition1})
by $[\cdot, \cdot]_{\mathfrak{m}_1}$.
By Lemma \ref{GO-criterion} we have the following:
A homogeneous Riemannian manifold   $(M=G_1/H_1,g)$ with the reductive
decomposition  (\ref{reductivedecomposition1}) is a geodesic orbit space if and
only if  for any $X_1 \in \mathfrak{m}_1$ there is $Z_1 \in
\mathfrak{h}_1$ such that
$([X_1+Z_1,Y_1]_{\mathfrak{m}_1},X_1)_1 =0$ for all $Y_1\in \mathfrak{m}_1$.

\medskip

Now we suppose that {\it a connected closed subgroup $G$ of the Lie group $G_1$ acts transitively
on the Riemannian manifold $(M,g)$}. Then $M$ is diffeomorphic to the homogeneous space $G/H$, where $H=G\cap H_1$.

There is an $\Ad(H)$-invariant decomposition
\begin{equation}\label{reductivedecomposition.n}
\mathfrak{g}=\mathfrak{h}\oplus \mathfrak{m},
\end{equation}
where $\mathfrak{g}={\rm Lie }(G)\subset \mathfrak{g}_1$ and $\mathfrak{h}={\rm Lie}(H)\subset \mathfrak{h}_1$.
The Riemannian metric $g$ is obviously $G$-invariant and is determined
by an $\Ad(H)$-invariant Euclidean metric $g = (\cdot,\cdot)$ on
the space $\mathfrak{m}$ which is also identified with the tangent
space $T_oM$ at the initial point $o = eH$.
\smallskip

Since $H$ is a compact subgroup of $H_1$, there is also
an $\Ad(H)$-invariant decomposition
\begin{equation}\label{reductivedecomposition.nn}
\mathfrak{h}_1=\mathfrak{h}\oplus \widetilde{\mathfrak{h}},
\end{equation}
in particular, $[\mathfrak{h}, \widetilde{\mathfrak{h}}] \subset \widetilde{\mathfrak{h}}$. Recall also that
$[\mathfrak{h}, \mathfrak{m}] \subset  \mathfrak{m}$ and $[\mathfrak{h}_1, \mathfrak{m}_1] \subset  \mathfrak{m}_1$.
\medskip

Since $G$ acts transitively on $M=G_1/H_1$, then we have the equality $\mathfrak{h}_1+\mathfrak{g}=\mathfrak{g}_1$ that is equivalent to the following equality
(as for linear spaces):
\begin{equation}\label{reductivedecomposition.nnn}
\mathfrak{g}_1=\mathfrak{h}_1\oplus \mathfrak{m}
\end{equation}
(this equality should not be an $\Ad(H_1)$-invariant decomposition).
By $[\cdot, \cdot]_{\mathfrak{m}}$ we will denote a $\mathfrak{m}$-component
of the Lie bracket of the Lie algebra $\mathfrak{g}_1$ according to decomposition (\ref{reductivedecomposition.nnn}).
Note, that for $X,Y \in  \mathfrak{g}$ this is consistent with the notation of $\mathfrak{m}$-component of $[X,Y]$ in the
decomposition~(\ref{reductivedecomposition.n}).
\smallskip

For any $X_1\in \mathfrak{m}_1$, we have the following unique decomposition by (\ref{reductivedecomposition.nn}) and  (\ref{reductivedecomposition.nnn}):
$$
X_1=X+W+\widetilde{W},
$$
where $X\in \mathfrak{m}$, $W\in \mathfrak{h}$ and $\widetilde{W}\in \widetilde{\mathfrak{h}}$. Now we can define three maps
$\pi:\mathfrak{m}_1 \rightarrow \mathfrak{m}$, $l:\mathfrak{m} \rightarrow \mathfrak{h}$, and
$\widetilde{l}:\mathfrak{m} \rightarrow \widetilde{\mathfrak{h}}$ as follows: we put $\pi (X_1)=X$, $l(X)=W$, and $\widetilde{l}(X)=\widetilde{W}$,
where $X_1=X+W+\widetilde{W}$ as above. It is easy to see that $\pi$ is bijective, therefore, the maps $l$ and $\widetilde{l}$ are correctly defined.
\smallskip

Since $(\cdot,\cdot)_1$ on
the space $\mathfrak{m}_1$ and $(\cdot,\cdot)$ on
the space $\mathfrak{m}$ generate one and the same Rimannian metric $g$ on $M$, we have the equality
\begin{equation}\label{innprod}
(X_1,Y_1)_1=(\pi(X_1),\pi(Y_1)), \quad  \quad X_1,Y_1 \in \mathfrak{m}_1.
\end{equation}
\medskip

Let us define one useful map $L:\widetilde{\mathfrak{h}} \times \mathfrak{m} \rightarrow \mathfrak{m}$ by the equality
\begin{equation}\label{vsmap}
L(U,X)=[U,X]_{\mathfrak{m}}\quad \mbox{(according to (\ref{reductivedecomposition.nnn}))}, \quad  \quad U\in \widetilde{\mathfrak{h}},\,X \in \mathfrak{m}.
\end{equation}
\medskip

Note that for a fixed $U$, the operator $L(U,\cdot):\mathfrak{m} \rightarrow \mathfrak{m}$ is skew-symmetric with respect to $(\cdot, \cdot)$. Indeed,
$$([U,X]_{\mathfrak{m}},X)=([U,X]_{\mathfrak{m}_1},X_1)_1=([U,X_1]_{\mathfrak{m}_1},X_1)_1=0, $$
where $X_1=\pi^{-1}(X)$, since the restriction od $\ad(U)$ on $\mathfrak{m}_1$ is skew-symmetric.
\medskip

Now, we are going to get the following

\begin{theorem}\label{nonfull1n}
A Riemannian manifold $(M=G_1/H_1,g)$ with $G_1$-invariant metric $g$ is a geodesic orbit space
{\rm(}with respect to the group $G_1${\rm)} if and only if for any $X \in \mathfrak{m}$ there are
$V\in \mathfrak{h}$ and $U \in \widetilde{\mathfrak{h}}$ such that for any $Y \in \mathfrak{m}$
the equality
$$
([X+V,Y]_{\mathfrak{m}},X)+(L(U,Y),X)=0
$$
holds {\rm(}see decomposition {\rm(\ref{reductivedecomposition.n}) and {\rm(\ref{vsmap})})}.
\end{theorem}

\begin{remark}
Let us emphasize that the $\mathfrak{m}$-component in the last equation can be considered according to the decomposition (\ref{reductivedecomposition.n}).
Formally, for $\mathfrak{g}=\mathfrak{g}_1$ the space $\widetilde{\mathfrak{h}}$ and the map $L$ are trivial, hence, in this partial case Theorem \ref{nonfull1n}
is equivalent to Lemma \ref{GO-criterion}.
\end{remark}

\begin{proof}
Suppose that $(M,g)$ is a geodesic orbit space.
Consider $X_1=\pi^{-1}(X)$ and $Y_1=\pi^{-1}(Y)$. By Lemma \ref{GO-criterion} there is $Z_1\in \mathfrak{h}_1$ such that
$([X_1+Z_1,Y_1]_{\mathfrak{m}_1},X_1)_1=0$. By (\ref{reductivedecomposition.nn}) we have $Z_1=Z+\widetilde{Z}$, where
$Z\in\mathfrak{h}$ and $\widetilde{Z}\in \widetilde{\mathfrak{h}}$. Moreover, by the above discussion we have
$$
X_1=X+l(X)+\widetilde{l}(X),\quad Y_1=Y+l(Y)+\widetilde{l}(Y).
$$
Note that
$$
0=([X_1+Z_1,Y_1]_{\mathfrak{m}_1},X_1)_1=(\pi([X_1+Z_1,Y_1]_{\mathfrak{m}_1}),\pi(X_1))=([X_1+Z_1,Y_1]_{\mathfrak{m}}, X),
$$
since $U_{\mathfrak{m}_1}-U_{\mathfrak{m}}\in \mathfrak{h}_1$ for any $U\in \mathfrak{g}_1$.
Further, we have (recall, that $[\mathfrak{h}, \mathfrak{m}] \subset  \mathfrak{m}$)
$$
[X_1+Z_1,Y_1]_{\mathfrak{m}}=[X+l(X)+Z+\widetilde{l}(X)+\widetilde{Z},Y+l(Y)+\widetilde{l}(Y)]_{\mathfrak{m}}=
$$
$$
[X,Y]_{\mathfrak{m}}+[l(X)+Z,Y]+[\widetilde{l}(X)+\widetilde{Z},Y]_{\mathfrak{m}}+[X,l(Y)]+[X,\widetilde{l}(Y)]_{\mathfrak{m}}.
$$
Since the restriction of $\ad(l(Y))$ on $\mathfrak{m}$ is skew-symmetric, we get $([X,l(Y)],X)=0$. Moreover,
the restriction of $\ad(\widetilde{l}(Y))$ on $\mathfrak{m}_1$ is also skew-symmetric, hence we get
$([X,\widetilde{l}(Y)]_{\mathfrak{m}},X)=([X,\widetilde{l}(Y)]_{\mathfrak{m}_1},X_1)_1=([X_1,\widetilde{l}(Y)]_{\mathfrak{m}_1},X_1)_1=0$.
Therefore, we have
$$
0=([X_1+Z_1,Y_1]_{\mathfrak{m}}, X)=([X,Y]_{\mathfrak{m}},X)+([l(X)+Z,Y],X)+([\widetilde{l}(X)+\widetilde{Z},Y]_{\mathfrak{m}},X)
$$
Putting $V:=l(X)+Z$ and $U:=\widetilde{l}(X)+\widetilde{Z}$, we get
$([X+V,Y]_{\mathfrak{m}},X)+(L(U,Y),X)=0$.
\smallskip

Conversely, suppose that for any $X \in \mathfrak{m}$ there are
$V\in \mathfrak{h}$ and $U \in \widetilde{\mathfrak{h}}$ such $([X+V,Y]_{\mathfrak{m}},X)+(L(U,Y),X)=0$ that for any $Y \in \mathfrak{m}$.
Consider the vectors $X_1=\pi^{-1}(X)$, $Y_1=\pi^{-1}(Y)$ and the decompositions
$X_1=X+l(X)+\widetilde{l}(X)$, $Y_1=Y+l(Y)+\widetilde{l}(Y)$.
If we put $Z:=V-l(X)$ and $\widetilde{Z}=U-\widetilde{l}(X)$, then we deduce (performing the above calculations in the reverse order) the equality
$$
0=([X_1+Z_1,Y_1]_{\mathfrak{m}}, X)=([X_1+Z_1,Y_1]_{\mathfrak{m}_1},X_1)_1,
$$
where $Z_1=Z+\widetilde{Z}$.
Since the map $\pi$ is bijective, by Lemma \ref{GO-criterion} we get that  $(M,g)$ is a  geodesic orbit space.
\end{proof}
\smallskip

In Theorem \ref{nonfull1n} one may consider the full connected isometry group of $(M, g)$ as $G_1$.
Hence we get

\begin{theorem}\label{nonfull1}
A Riemannian homogeneous manifold $(M,g)$ is a  geodesic orbit manifold if and only if for any $X \in \mathfrak{m}$ there are
$V\in \mathfrak{h}$  and $U \in \widetilde{\mathfrak{h}}$ {\rm(}see {\rm(\ref{reductivedecomposition.nn}))} such that for any $Y \in \mathfrak{m}$
the equality
$$
([X+V,Y]_{\mathfrak{m}},X)+(L(U,Y),X)=0
$$
holds {\rm(}see  and decomposition {\rm(\ref{reductivedecomposition.n})} and {\rm(\ref{vsmap})},
where $G_1$ is the full connected isometry group of $(M,g)${\rm)}.
\end{theorem}

Since the operator $\ad(V):\mathfrak{m} \rightarrow \mathfrak{m}$ is skew-symmetric and
the operator $L(U,\cdot):\mathfrak{m} \rightarrow \mathfrak{m}$ is skew-symmetric (with respect to $(\cdot, \cdot)$) for any fixed $U$, we get

\begin{corollary}\label{simplecor}
If a Riemannian manifold $(M=G_1/H_1,g)$ is geodesic orbit, then for any $X \in \mathfrak{m}$ there is
a skew-symmetric {\rm(}with respect to $g=(\cdot, \cdot)${\rm)} endomorphism $L :\mathfrak{m} \rightarrow \mathfrak{m}$
such that for any $Y \in \mathfrak{m}$
the equality
$([X,Y]_{\mathfrak{m}},X)+(L(Y),X)=0$
holds.
\end{corollary}

\medskip

We will use the notations $\mathfrak{n}_{\mathfrak{g}}(\mathfrak{h})$, $ \mathfrak{c}_{\mathfrak{g}}(\mathfrak{h})$, and $ \mathfrak{c}(\mathfrak{h})$
for the normalizer of $\mathfrak{h}$ in $\mathfrak{g}$, the centralizer of $\mathfrak{h}$ in $\mathfrak{g}$, and the center of $\mathfrak{h}$ respectively.
It is clear that $\mathfrak{c}(\mathfrak{h})=\mathfrak{c}_{\mathfrak{g}}(\mathfrak{h}) \cap \mathfrak{h}$.
It is easy to see the following: if $W \in \mathfrak{c}_{\mathfrak{g}}(\mathfrak{h})$ and $W=W_1+W_2$, where
$W_1 \in  \mathfrak{m}$ and $W_2 \in  \mathfrak{h}$, then $W_1,W_2 \in \mathfrak{c}_{\mathfrak{g}}(\mathfrak{h})$.

\begin{prop}\label{nonfull2}
Let $(M=G/H,g)$ be a  homogeneous Riemannian space with reductive decomposition {\rm(\ref{reductivedecomposition.n})} and
let $K$ be a compact subgroup of $G$ such that $K \cap H =\{e\}$ and $\mathfrak{k}:=\Lie(K)\subset \mathfrak{n}_{\mathfrak{g}}(\mathfrak{h})$.
Suppose also that the inner product $(\cdot,\cdot)$ is $\Ad(K)$-invariant.
If
for any $X \in \mathfrak{m}$ there are
$V\in \mathfrak{h}$ and $W\in \mathfrak{k}$
such that for any $Y \in \mathfrak{m}$
the equality
$$
([X+V+W,Y]_{\mathfrak{m}},X)=0
$$
holds, then $(M,g)$ is a geodesic orbit space with respect the group $G\times K$ that acts isometrically on $(M,g)$ via $(a,b) (cH)=acHb^{-1}$.
\end{prop}

\begin{proof} It is easy to see that $[ \mathfrak{k}, \mathfrak{h}]=0$.
Since $(\cdot,\cdot)$ is $\Ad(K)$-invariant, the the group $G\times K$ acts isometrically on $(M,g)$ via $(a,b) (cH)=acHb^{-1}$.
The Lie algebra of the group $G_1:=G\times K$ has the form $\mathfrak{g}_1=\mathfrak{g}\oplus \mathfrak{k}$,  and
$\mathfrak{h}_1=\mathfrak{h}\oplus \diag(\mathfrak{k})$ (here $\mathfrak{h}$ is supposed embedded in $\mathfrak{g}_1=\mathfrak{g}\oplus \mathfrak{k}$
via $X \mapsto (X,0)$). Hence, (in the above notations) we can choose $\widetilde{\mathfrak{h}}:=\diag(\mathfrak{k})$
(see (\ref{reductivedecomposition.nn})).

By condition of the proposition, for any $X \in \mathfrak{m}$ there are
$V\in \mathfrak{h}$ and $W \in  \mathfrak{c}_{\mathfrak{g}}(\mathfrak{h})$ such that $([X+V,Y]_{\mathfrak{m}},X)+([W,Y]_{\mathfrak{m}},X)=0$
for any $Y \in \mathfrak{m}$. If we take $U=(W,W)\in \mathfrak{g}\oplus \mathfrak{k} =\mathfrak{g}_1$, then obviously we get
$([X+V,Y]_{\mathfrak{m}},X)+(L(U,Y),X)=0$. Therefore,  applying Theorem~\ref{nonfull1n}, we see that
$(M,g)$ is a GO-space with respect $G_1=G\times K$.
\end{proof}

\medskip

Now we consider more closely a special case when  $(M,g)$ is a compact homogeneous Riemannian manifold,
and the normal subgroup $G$ of the full connected isometry group $G_1$ of $(M,g)$ acts transitively on $M$.

\smallskip

In the case when $G/H$ is compact in Proposition \ref{nonfull2}, we can reformulate the assertion of this proposition in terms
of a metric endomorphism.
The (compact) Lie algebra $\mathfrak{g}$ admits a bi-invariant inner product $\langle \cdot ,\cdot \rangle$. It is easy to see that there is
$\Ad(H)$-equivariant linear operator $A:\mathfrak{m} \rightarrow \mathfrak{m}$ such that $(\cdot,\cdot)=\langle A \cdot, \cdot \rangle|_{\mathfrak{m}}$.
Obviously, this operator is non-degenarate and symmetric. It is called a {\it metric endomorphism}.

\begin{prop}\label{nonfull4} Let $(M=G/H,g)$ be a compact homogeneous Riemannian space with reductive decomposition {\rm(\ref{reductivedecomposition.n})} and
let $K$ be a closed subgroup of $G$ such that $K \cap H =\{e\}$, $\mathfrak{k}:=\Lie(K)\subset \mathfrak{n}_{\mathfrak{g}}(\mathfrak{h})$,
the inner product $(\cdot,\cdot)$ is $\Ad(K)$-invariant.
Then the following conditions are equivalent:

{\rm 1)} For any $X \in \mathfrak{m}$ there are
$V\in \mathfrak{h}$ and $W\in \mathfrak{c}_{\mathfrak{g}}(\mathfrak{h}) \cap \mathfrak{m}$
such that for any $Y \in \mathfrak{m}$
the equality
$([X+V+W,Y]_{\mathfrak{m}},X)=0$
holds.

{\rm 2)} For any $X \in \mathfrak{m}$ there are
$V\in \mathfrak{h}$ and $W\in \mathfrak{c}_{\mathfrak{g}}(\mathfrak{h}) \cap \mathfrak{m}$
such that
$[A(X),X+V+W]\in \mathfrak{h}$,
where $A:\mathfrak{m} \rightarrow \mathfrak{m}$ is a metric endomorphism.
\end{prop}

\begin{proof} Using the fact that  $\langle \cdot ,\cdot \rangle$ is bi-invariant inner product, we get
$$
([X+V+W,Y]_{\mathfrak{m}},X)=\langle [X+V+W,Y]_{\mathfrak{m}},A(X) \rangle=
$$
$$
\langle [X+V+W,Y],A(X) \rangle=-\langle Y,[X+V+W,A(X)] \rangle\,,
$$
that proofs the proposition (recall that $Y\in \mathfrak{m}$ assumed to be arbitrary).
\end{proof}
\smallskip

As a corollary of Theorem \ref{nonfull1}, Propositions \ref{nonfull2} and \ref {nonfull4} we get

\begin{theorem}\label{nonfull3}
Let $(M,g)$ be a compact homogeneous Riemannian manifold with the full connected isometry group $G_1$, and let $G$ be a normal subgroup of $G_1$,
that acts transitively on $M$. Then $(M,g)$ is a geodesic orbit manifold if and only if
for any $X \in \mathfrak{m}$ there are
$V\in \mathfrak{h}$ and $W\in \mathfrak{c}_{\mathfrak{g}}(\mathfrak{h})\cap \mathfrak{m}$ such that for any $Y \in \mathfrak{m}$
the equality
$([X+V+W,Y]_{\mathfrak{m}},X)=0$
holds {\rm(}see {\rm(\ref{reductivedecomposition.n}))}. It is equivalent also to the following: For any $X \in \mathfrak{m}$ there are
$V\in \mathfrak{h}$ and $W\in \mathfrak{c}_{\mathfrak{g}}(\mathfrak{h}) \cap \mathfrak{m}$
such that
$[A(X),X+V+W]\in \mathfrak{h}$,
where $A:\mathfrak{m} \rightarrow \mathfrak{m}$ is the metric endomorphism.
\end{theorem}

\begin{proof}
Under conditions of the theorem, there is an ideal $\mathfrak{a}$ in $\mathfrak{g}_1$ such that
$\mathfrak{g}_1=\mathfrak{a} \oplus \mathfrak{g}$ as Lie algebras. In this case every element $X\in \widetilde{\mathfrak{h}}$
{\rm(}see {\rm(\ref{reductivedecomposition.nn}))}
we may uniquely represent as $X=A+B$, where $A \in \mathfrak{a}$ and $B\in \mathfrak{g}$. Therefore, we get two maps
(in fact, Lie algebra homomorphisms)
$\sigma_1:\widetilde{\mathfrak{h}} \rightarrow \mathfrak{a}$ and $\sigma_2:\widetilde{\mathfrak{h}} \rightarrow \mathfrak{g}$ acting via
$A=\sigma_1(X)$ and $B=\sigma_2(X)$.

Since $[\mathfrak{h},\widetilde{\mathfrak{h}}] \subset \widetilde{\mathfrak{h}}$, then for every $Y\in \mathfrak{h}$ we have $[Y,X]\in \widetilde{\mathfrak{h}}$.
On the other hand, $[Y,X]=[Y,\sigma_1(X)+\sigma_2(X)]=[Y,\sigma_2(X)]\in \mathfrak{g}$. Therefore, $[Y,X]=[Y,\sigma_2(X)]=0$.
In particular, the image of the map $\sigma_2$ lies in $\mathfrak{c}_{\mathfrak{g}}(\mathfrak{h})$ (the centralizer of $\mathfrak{h}$ in $\mathfrak{g}$)
and $[\mathfrak{h},\widetilde{\mathfrak{h}}]=0$. From the last equality we see that $\widetilde{\mathfrak{h}}$ is a Lie algebra, and
$\sigma_1$, $\sigma_2$ are really Lie algebra homomorphisms.

It is easy to see that $\sigma_1$ is bijective, because $\dim \mathfrak{a} =\dim \widetilde{\mathfrak{h}}$ and
$\Ker (\sigma_1) \subset \mathfrak{g}\cap \widetilde{\mathfrak{h}}=\{0\}$.
Hence, $\tau=\sigma_2\circ \sigma_1^{-1}:\mathfrak{a}\rightarrow \mathfrak{g}$
is a well defined homomorphism of Lie algebras, $\tau(\mathfrak{a})$ is a homomorphic image of $\mathfrak{a}$,
and $\mathfrak{a}$ is isomorphic to a direct sum of $\tau(\mathfrak{a})$ and $\tau^{-1}(0)$.
Note that $\tau^{-1}(0)$ is trivial, since $\tau^{-1}(0)$ is an ideal in $\mathfrak{g}_1$ and $\tau^{-1}(0) \subset \mathfrak{h}_1$
(recall, that we have an effective action of $G_1$ on $M$). Therefore, $\tau(\mathfrak{a})$ is a isomorphic to $\mathfrak{a}$ and also
to $\widetilde{\mathfrak{h}}$.

Since $\tau(\mathfrak{a})=\sigma_2(\widetilde{\mathfrak{h}})$ we have $[\mathfrak{h},\tau(\mathfrak{a})]=0$.
It is easy to see also that $\mathfrak{h} \cap \tau(\mathfrak{a})$ is trivial.
\smallskip

Now, let us suppose that $(M,g)$ is a geodesic orbit manifold and fix $X\in\mathfrak{m}$. By
Theorem~\ref{nonfull1}, there are
$V\in \mathfrak{h}$ and $U \in \widetilde{\mathfrak{h}}$ such that $([X+V,Y]_{\mathfrak{m}},X)+(L(U,Y),X)=0$ for any $Y \in \mathfrak{m}$.

From the above discussion we get $U=T+\tau(T)$ for some $T\in \mathfrak{a}$. Since $[\mathfrak{a},\mathfrak{g}]=0$, then
$[U,Y]_{\mathfrak{m}}=[\tau(T),Y]_{\mathfrak{m}}$. If $\tau(T)\in \mathfrak{m}$, then we can take $W:=\tau(T)$.
Hence we get $([X+V+W,Y]_{\mathfrak{m}},X)=0$. But in general $\tau(T)\not\in \mathfrak{m}$.
In any case we have
$\tau(T)=W_1+W_2$, where
$W_1 \in  \mathfrak{m}$ and $W_2 \in  \mathfrak{h}$. Since $\tau(T) \in \mathfrak{c}_{\mathfrak{g}}(\mathfrak{h})$,
then $W_1,W_2 \in \mathfrak{c}_{\mathfrak{g}}(\mathfrak{h})$. We may take $W:=W_1$ and change $V$ to $V+W_2 \in \mathfrak{h}$.
Hence, we again get $([X+V+W,Y]_{\mathfrak{m}},X)=0$.
\smallskip

The opposite implication follows directly from Proposition \ref{nonfull2}. The final assertion is valid due to Proposition \ref {nonfull4}.
\end{proof}
\smallskip

In the next two sections we will apply Theorem \ref{nonfull3} for the study invariant Riemannian metric on the
Aloff--Wallach spaces and left-invariant metrics on compact Lie groups.

\section{Geodesic orbit Riemannian metric on the Aloff--Wallach spaces}\label{AWspaces}

In this section we describe geodesic orbit Riemannian metrics on generic Aloff--Wallach spaces.
The Aloff--Wallach spaces
$W_{k,l}=SU(3)/SO(2)$ are defined
by an embedding of the circle $SO(2)=S^1$ into $SU(3)$ of the type
$$
i_{k,l} : e^{2\pi {\bf i} \theta} \mapsto \diag
(e^{2\pi {\bf i} k \theta},e^{2\pi {\bf i} l\theta},e^{2\pi {\bf i} m\theta}),
$$
where $k$, $l$, $m$ are integers with greatest common divisor $1$
and $k+l+m=0$, ${\bf i}=\sqrt{-1}$. By using the Weil group of
$SU(3)$ one can assume that $k\geq l\geq 0$. These spaces were
studied by S.~Aloff and N.R.~Wallach in \cite{Alof}, where
they showed that $W_{k,l}$  admits an invariant metric  of
positive sectional curvature if and only if $kl(k+l)\neq 0$. Moreover,
$H^4(W_{k,l};\mathbb{Z})=\mathbb{Z}/|k^2+l^2+kl|\mathbb{Z}$,
and hence there are infinitely
many different homotopy types among these spaces. Later M.~Kreck
and S.~Stolz found in \cite{Kreck} that there are homeomorphic but
non-diffeomorphic spaces among the $W_{k,l}$.
Below we list some useful facts about the Aloff-Wallach spaces (see e.~g. \cite{Alof, Kreck, Nik2004}).

Let $\mathfrak{h}=\mathfrak{h}_{k,l}$ be the Lie algebra of the Lie group
$i_{k,l}(S^1)=H_{k,l}=H$ and let $\mathfrak{t}$ be the Lie algebra of
the standard maximal torus $T$ in $SU(3)$.
It is useful to define the value $L=k^2+l^2+m^2$ and to check that
$k^2+l^2+m^2-kl-km-ml=3L/2$.

Let us fix the inner product $\langle X,Y \rangle =-\frac{1}{2} \operatorname{Re} \trace (XY)$
on the Lie algebra $su(3)$.
It is easy to check that that $B(X,Y)= -12\cdot \langle X,Y \rangle$
for the Killing form $B$ of Lie algebra $su(3)$.

Let us consider the following vectors
$$
Z={\bf i}
\left(
\begin{array}{ccc}
k&0&0\\
0&l&0\\
0&0&m\\
\end{array}
\right),
\quad \quad \quad \quad
X_0=\frac{\sqrt{2}{\bf i}}{\sqrt{3L}}
\left(
\begin{array}{ccc}
l-m&0&0\\
0&m-k&0\\
0&0&k-l\\
\end{array}
\right),
$$
$$
X_1=\left(
\begin{array}{ccc}
0&1&0\\
-1&0&0\\
0&0&0\\
\end{array}
\right),
\quad
X_2=\left(
\begin{array}{ccc}
0&{\bf i}&0\\
{\bf i}&0&0\\
0&0&0\\
\end{array}
\right),
\quad
X_3=\left(
\begin{array}{ccc}
0&0&1\\
0&0&0\\
-1&0&0\\
\end{array}
\right),
$$
$$
X_4=\left(
\begin{array}{ccc}
0&0&{\bf i}\\
0&0&0\\
{\bf i}&0&0\\
\end{array}
\right),
\quad
X_5=\left(
\begin{array}{ccc}
0&0&0\\
0&0&1\\
0&-1&0\\
\end{array}
\right),
\quad
X_6=\left(
\begin{array}{ccc}
0&0&0\\
0&0&{\bf i}\\
0&{\bf i}&0\\
\end{array}
\right)
$$
in the Lie algebra $su(3)$.
Note, that the subalgebra $\mathfrak{h}$ is defined by the vector $Z$.
Moreover,
all the vectors
$X_i$ have unit length with respect
to the  inner product $\langle \cdot,\cdot \rangle$,
are mutually orthogonal and are orthogonal to the subalgebra $\mathfrak{h}$.
Let us consider the $\Ad (H)$-modules
$\mathfrak{m}_1=\Lin (X_1,X_2)$,
$\mathfrak{m}_2=\Lin (X_3,X_4)$,
$\mathfrak{m}_3=\Lin (X_5,X_6)$,
$\mathfrak{m}_4=\Lin (X_0)$.

One can obtain the following decomposition
$$
\mathfrak{g}=\mathfrak{t} \oplus \mathfrak{m}_1 \oplus
\mathfrak{m}_2 \oplus \mathfrak{m}_3=
\mathfrak{h}\oplus \mathfrak{m}_4 \oplus \mathfrak{m}_1 \oplus \mathfrak{m}_2
\oplus \mathfrak{m}_3,
$$
i.~e.
$\mathfrak{m}_4$ is the $\langle \cdot,\cdot \rangle$-orthogonal
complement to $\mathfrak{h}=\mathfrak{h}_{k,l}$ in the Lie
algebra $\mathfrak{t}$,
and $\mathfrak{m}=\mathfrak{m}_1 \oplus \mathfrak{m}_2
\oplus \mathfrak{m}_3 \oplus \mathfrak{m}_4$
is the orthogonal complement to
$\mathfrak{h}=\mathfrak{h}_{k,l}$ in $su(3)$.
Direct calculations show that
$[Z,X_0]=0$,
$[Z,X_1]=(k-l)X_2$,
$[Z,X_2]=(l-k)X_1$,
$[Z,X_3]=(k-m)X_4$,
$[Z,X_4]=(m-k)X_3$,
$[Z,X_5]=(l-m)X_6$,
$[Z,X_6]=(m-l)X_5$.

It easy to check that the modules $\mathfrak{m}_i$
are $\Ad (H)$-invariant and
$\Ad (H)$-irreducible for
pairwise distinct
$k$, $l$ and $m$.
Moreover, there is no
pairwise isomorphic
$\Ad (H)$-modules among $\mathfrak{m}_i$, if $(k,l,m)\neq (1,1,-2)$ or $(k,l,m)\neq (1,-1,0)$
(see details in \cite{Nik2004}).
These implies (see e.~g. \cite{Bes}) that every $\Ad (H)$-invariant
inner product $(\cdot , \cdot )$ (i.~e. every $SU(3)$-invariant Riemannian metric on $W_{k,l}$)
has the following form
\begin{equation}\label{awmet}
(\cdot , \cdot )=
x_1  \langle \cdot,\cdot \rangle|_{{\mathfrak{m}}_1}+
x_2  \langle \cdot,\cdot \rangle|_{\mathfrak{m}_2}+
x_3  \langle \cdot,\cdot \rangle|_{\mathfrak{m}_3}+
x_4  \langle \cdot,\cdot \rangle|_{\mathfrak{m}_4}
\end{equation}
for some positive
$x_i$.

\begin{remark}
$W_{1,0}$ and $W_{1,1}$ are  special in the following sense: all other
$W_{k,l}$ has a 4-parameter family of invariant metrics. But on
$W_{1,0}$ and $W_{1,1}$ the set of invariant metrics depends on 6
and 10 parameters respectively, see details in \cite{Nik2004}.
\end{remark}

The most important fact is that the full connected isometry group of every $SU(3)$-invariant Riemannian metric on $W_{k,l}$
for  $(k,l)\neq (1,0)$ and  $(k,l)\neq (1,1)$ is locally isomorphic to $U(3)$ (see e.~g. \cite{Shan}), hence $SU(3)$ is a normal subgroup in this group.
Therefore, we may apply Theorem~\ref{nonfull3} to these homogeneous Riemannian manifolds.
\smallskip

\begin{prop} Geodesic orbit Riemannian metrics on the Aloff--Wallach space $W_{k,l}$, where
$(k,l)\neq (1,0)$ and  $(k,l)\neq (1,1)$, are exactly $SU(3)$-invariant metrics {\rm (\ref{awmet})} with $x_1=x_2=x_3$.
\end{prop}

\begin{proof}
We will use Theorem \ref{nonfull3} in our proof.
In the basis $(X_0,X_1,\dots,X_6)$ the metric endomorphism for $(\cdot,\cdot)$ (with respect to $\langle X,Y \rangle =-\frac{1}{2} \operatorname{Re} \trace (XY)$)
has the form
$A=\diag(x_4, x_1,x_1, x_2,x_2,x_3,x_3)$. For an arbitrary vector
$X=\sum_{i=0}^6 \alpha_i X_i$, we get
$$
A(X)=x_4 \alpha_0 X_0+x_1\alpha_1 X_1+x_1\alpha_2 X_2+x_2\alpha_3 X_3+x_2\alpha_4 X_4+x_3\alpha_5 X_5+x_3\alpha_6 X_6.
$$
It is clear that $\mathfrak{h}=\Lin(Z)$ and $\mathfrak{c}_{\mathfrak{g}}(\mathfrak{h})\cap \mathfrak{m}=\Lin(X_0)$. Therefore, any vector
$V+W$ (in the notation of Theorem \ref{nonfull3}) has the form
$$
V+W={\bf i}\diag(\beta,\gamma,-\beta-\gamma), \quad \beta,\gamma \in \mathbb{R}.
$$
Using Lie multiplication we can represent the vector $[A(X),X+V+W]$
in the base $(Z,X_0,X_1,\dots,X_6)$. Hence the condition $[A(X),X+V+W]\in \mathfrak{h}=\Lin(Z)$ could be represented
as the system of $7$ linear equations ($\langle [A(X),X+V+W],X_i \rangle=0$, $i=0,1,\dots,6$)
with respect to the variables $\beta$ and $\gamma$ with the parameters
$x_1,x_2,x_3,x_4$ (fixing the metric) and $\alpha_i$, $i=0,...,6$, (fixing the vector $X$).

Hence, a Riemannian manifolds $(W_{k,l},g=(\cdot,\cdot))$ is a geodesic orbit manifold if and only if {\it the above system has a solution for all
sets of parameters}.
Direct calculations show that
$$
X=\left(
\begin{array}{rrr}
\alpha_0(l-m)f {\bf i}&\alpha_1+\alpha_2 {\bf i}&\alpha_3+\alpha_4 {\bf i}\\
-\alpha_1+\alpha_2 {\bf i}&\alpha_0(m-k)f {\bf i}&\alpha_5+\alpha_6 {\bf i}\\
-\alpha_3+\alpha_4 {\bf i}&-\alpha_5+\alpha_6 {\bf i}&\alpha_0(k-l)f {\bf i}\\
\end{array}
\right)\,,
$$
$$
X+V+W=\left(
\begin{array}{rrr}
(\alpha_0(l-m)f +\beta){\bf i}&\alpha_1+\alpha_2 {\bf i}&\alpha_3+\alpha_4 {\bf i}\\
-\alpha_1+\alpha_2 {\bf i}&(\gamma+\alpha_0(m-k)f) {\bf i}&\alpha_5+\alpha_6 {\bf i}\\
-\alpha_3+\alpha_4 {\bf i}&-\alpha_5+\alpha_6 {\bf i}&(-\beta-\gamma+\alpha_0(k-l)f) {\bf i}\\
\end{array}
\right)\,,
$$
$$
A(X)=\left(
\begin{array}{rrr}
x_4\alpha_0(l-m)f {\bf i}&x_1(\alpha_1+\alpha_2 {\bf i})&x_2(\alpha_3+\alpha_4 {\bf i})\\
x_1(-\alpha_1+\alpha_2 {\bf i})&x_4\alpha_0(m-k)f {\bf i}&x_3(\alpha_5+\alpha_6 {\bf i})\\
x_2(-\alpha_3+\alpha_4 {\bf i})&x_3(-\alpha_5+\alpha_6 {\bf i})&x_4\alpha_0(k-l)f {\bf i}\\
\end{array}
\right)\,,
$$
where $f=\sqrt{\frac{2}{3L}}$. The condition to be a geodesic orbit metric is equivalent to the following condition for matrices:
{\it the matrix $[A(X),X+V+W]$ should be a multiple of the matrix $Z={\bf i}\diag(k,l,-k-l)$.}
Let us use the notation $U=(u_{ij})$ for the matrix  $[A(X),X+V+W]$. It is easy to see that $u_{11}=u_{22}=u_{33}=0$.
Therefore, $U=0$ is equivalent to the condition to be a geodesic orbit manifold. We easily to check the following:
$$
\frac{\alpha_2}{2{\bf i}}(u_{12}+u_{21})+\frac{\alpha_1}{2}(u_{12}-u_{21})=({x_{2}} - {x_{3}})\,({\alpha _{2}}\,{\alpha _{4}}\,{
\alpha _{5}} - {\alpha _{2}}\,{\alpha _{3}}\,{\alpha _{6}} + {\alpha _{1}}\,{\alpha _{3}}\,{\alpha _{5}} + {\alpha _{1}}\,{\alpha _{4}}\,{\alpha _{6}}),
$$
$$
\frac{\alpha_4}{2{\bf i}}(u_{13}+u_{31})+\frac{\alpha_3}{2}(u_{13}-u_{31})= ( {x_{3}} - {x_{1}})\,({\alpha _{2}}\,{\alpha _{
4}}\,{\alpha _{5}} - {\alpha _{2}}\,{\alpha _{3}}\,{\alpha _{6}}
 + {\alpha _{1}}\,{\alpha _{3}}\,{\alpha _{5}} + {\alpha _{1}}\,{
\alpha _{4}}\,{\alpha _{6}}),
$$
$$
\frac{\alpha_6}{2{\bf i}}(u_{23}+u_{32})+\frac{\alpha_5}{2}(u_{23}-u_{32})=( {x_{1}} - {x_{2}})\,({\alpha _{2}}\,{\alpha _{4}}
\,{\alpha _{5}} - {\alpha _{2}}\,{\alpha _{3}}\,{\alpha _{6}} + {
\alpha _{1}}\,{\alpha _{3}}\,{\alpha _{5}} + {\alpha _{1}}\,{
\alpha _{4}}\,{\alpha _{6}}).
$$
Hence, if $(W_{k,l},g=(\cdot,\cdot))$ is a geodesic orbit manifold, where $(k,l)\neq (1,0)$ and  $(k,l)\neq (1,1)$, then $x_1=x_2=x_3$.

On the other hand, for $x_1=x_2=x_3=:x$ we can choose
$$
\beta=(x_4/x-1)\alpha_0(l-m)f, \quad \gamma=(x_4/x-1)\alpha_0(m-k)f.
$$
It is easy to see that $X+V+W=x \cdot A(X)$ and, hence,  $U=[A(X),X+V+W]=0$ for such $\beta$ and $\gamma$.
Therefore, $(W_{k,l},g=(\cdot,\cdot))$ is a geodesic orbit manifold for $x_1=x_2=x_3$. Note also that this metric is $SU(3)$-normal homogeneous
(hence $SU(3)$-naturally reductive) for $x_1=x_2=x_3=x_4$ and is $U(3)$-naturally reductive for $x_1=x_2=x_3\neq x_4$.
\end{proof}

\section{Geodesic orbit Riemannian metrics on simple compact Lie groups}

By the Ochiai--Takahashi theorem \cite{OT1976}, the full connected isometry group $\Isom(G,\rho)$
of a simple compact Lie group $G$ with a left-invariant Riemannian metric $\rho$
contains in the group $L(G)R(G)$, the product of left and right translations. Hence $G$ is a normal subgroup in $\Isom(G,\rho)$,
that is locally isomorphic to the group
$G\times K$, where $K$ is a closed subgroup of $G$, with action
$(a,b) (c)=acb^{-1}$, where $a,c \in G$ and $b\in K$.

Consider a bi-invariant inner product $\langle \cdot, \cdot \rangle=-B(\cdot,\cdot)$ on $\mathfrak{g}$. It is easy to see that the metric $\rho$
is generated with an inner product of the following type

\begin{equation}\label{innersimpl1}
(\cdot,\cdot)= x_1\langle \cdot, \cdot \rangle|_{\mathfrak{k}_1}+\cdots+
x_p\langle \cdot, \cdot \rangle|_{\mathfrak{k}_p} +
y_1\langle \cdot, \cdot \rangle|_{\mathfrak{m}_1}+\cdots+
y_q\langle \cdot, \cdot \rangle|_{\mathfrak{m}_q},
\end{equation}
where $\mathfrak{k}=\mathfrak{k}_1 \oplus \cdots \oplus\mathfrak{k}_p$ is a decomposition in simple and 1-dimensional ideals,
$\mathfrak{m}$ is the $\langle \cdot, \cdot \rangle$-orthogonal complement to $\mathfrak{k}$ in $\mathfrak{g}$, and
$\mathfrak{m}=\mathfrak{m}_1 \oplus \cdots \oplus  \mathfrak{m}_q$ is a decomposition into the sum of $\Ad(K)$-invariant and
$\Ad(K)$-irreducible modules.
Note that the coresponding metric endomorphism is of the form $A=\diag(x_1,\dots,x_p, y_1,\dots, y_q)$ on
$\mathfrak{g}=\mathfrak{k}\oplus\mathfrak{m}=\mathfrak{k}_1 \oplus \cdots \oplus\mathfrak{k}_p \oplus \mathfrak{m}_1 \oplus \cdots \oplus  \mathfrak{m}_q$.
\smallskip

From Theorem \ref{nonfull3}, we easily get

\begin{prop}\label{nonfull5} A simple compact Lie group $G$ with a left-invariant Riemannian metric $\rho$
is a geodesic orbit manifold if and only if there is a closed connected subgroup $K$ of $G$ such that
for any $X \in \mathfrak{g}$ there is
$W\in \mathfrak{k}$
such that for any $Y \in \mathfrak{g}$
the equality
$([X+W,Y],X)=0$
holds or, equivalently, $[A(X),X+W]=0$,
where $A:\mathfrak{g} \rightarrow \mathfrak{g}$ is a metric endomorphism.
\end{prop}
\medskip

Consider the following inner product on the Lie algebra $\mathfrak{g}$ of a compact Lie group $G$:

\begin{equation}\label{lieego1}
(\cdot, \cdot)=u_1 \langle \cdot, \cdot \rangle |_{\mathfrak{p}_1}+u_1 \langle \cdot, \cdot \rangle |_{\mathfrak{p}_2}+ \cdots +
u_s \langle \cdot, \cdot \rangle |_{\mathfrak{p}_s},
\end{equation}
where $\mathfrak{g}=\mathfrak{p}_1\oplus \mathfrak{p}_2\oplus \cdots \oplus \mathfrak{p}_s$ is a $B$-orthogonal decomposition and
$u_1,\dots, u_s \in {\mathbb R}^+$ (in fact this means that we have the eigen-decomposition of the metric endomorphism $A$ defined by
$(\cdot, \cdot)=\langle A \cdot, \cdot \rangle$).
Let us find conditions for (\ref{lieego1}) to be a geodesic orbit metric.
We will call a Lie subalgebra $\mathfrak{k}\subset \mathfrak{g}$ {\it adapted} for (\ref{lieego1}), if
$\mathfrak{k}$ is direct sum of its ideals $\mathfrak{k} \cap \mathfrak{p}_i$, $i=1,\dots,s$, (some of these ideals could be trivial)
and the $B$-orthogonal complement to $\mathfrak{k} \cap \mathfrak{p}_i$ in $\mathfrak{p}_i$ is $\ad(\mathfrak{k})$-invariant for every $i=1,\dots,s$.
In particular, this implies that the metric endomorphism $A$ is $\ad(\mathfrak{k})$-equivariant.
It is clear that there is a maximal by inclusion adapted subalgebra among all subalgebras adopted for (\ref{lieego1}).
\smallskip

From  Proposition \ref{nonfull5} we get the following

\begin{prop}\label{nonfull6} The inner product {\rm (\ref{lieego1})} generates a geodesic orbit left-invariant Riemannian metric on $G$ if and only if
there is a maximal by inclusion adapted Lie subalgebra $\mathfrak{k}$ such that
for any $X \in \mathfrak{g}$ there is
$W\in \mathfrak{k}$
such that $[A(X),X+W]=0$,
where $A:\mathfrak{g} \rightarrow \mathfrak{g}$ is a metric endomorphism.
\end{prop}

\begin{proof}
Suppose that (\ref{lieego1}) generates a geodesic orbit left-invariant Riemannian metric on $G$.
Clear that the inner product (\ref{lieego1}) should coincides with the inner product (\ref{innersimpl1}), where
$G\times K$ is (at least locally) a full isometry group of the metric $\rho$ generated with (\ref{lieego1}).
Further, every $\mathfrak{k}_j$ in (\ref{innersimpl1}) should coincides with some $\mathfrak{k} \cap \mathfrak{p}_i$, $i=1,\dots,s$, in (\ref{lieego1}).
It is clear also that every $\mathfrak{m}_j$ in (\ref{innersimpl1}) should be a subset of one of $\mathfrak{p}_i$, $i=1,\dots,s$, in particular,
the $B$-orthogonal complement to $\mathfrak{k} \cap \mathfrak{p}_i$ in $\mathfrak{p}_i$ is $\ad(\mathfrak{k})$-invariant for every $i=1,\dots,s$.
Hence, $\mathfrak{k}$ is adapted for (\ref{lieego1}). By Proposition \ref{nonfull5}, we get that
any $X \in \mathfrak{m}$ there is
$W\in \mathfrak{k}$ such that $[A(X),X+W]=0$ holds. Obviously, the same is true for any adapted extension of $\mathfrak{k}$, hence,
for some maximal by inclusion adapted Lie subalgebra.

Now, suppose that we  have an adapted Lie subalgebra $\mathfrak{k}$ such that
for any $X \in \mathfrak{m}$ there is
$W\in \mathfrak{k}$
such that $[A(X),X+W]=0$. Let us consider $\mathfrak{m}_i$, the $B$-orthogonal complement to
$\mathfrak{k}_i:= \mathfrak{k} \cap \mathfrak{p}_i$, $i=1,\dots,s$.  By definition, every $\mathfrak{m}_i$ is $\ad(\mathfrak{k})$-invariant.
Therefore, the inner product (\ref{lieego1}) has the form (\ref{innersimpl1}), where $K=\exp(\mathfrak{k}) \subset G$.
Clear that this property saves if we consider more extended adapted Lie subalgebra $\mathfrak{k}$.

Let us take a maximal by inclusion adapted Lie subalgebra $\mathfrak{k}$ such that
for any $X \in \mathfrak{g}$ there is
$W\in \mathfrak{k}$ such that $[A(X),X+W]=0$. First, let us prove that the group $K=\exp(\mathfrak{k})$ is closed in $G$.
Really, if $\mathfrak{m}_i$ is $\ad(\mathfrak{k})$-invariant then it is $\Ad(\exp(\mathfrak{k}))$-invariant
and hence $\Ad(\overline{\exp(\mathfrak{k})})$-invariant. Moreover, $\mathfrak{k}$  is an ideal in $\overline{\mathfrak{k}}:=\Lie (\overline{\exp(\mathfrak{k})})$.
By the same reason, the metric endomorphism $A$ is $\Ad(\overline{\exp(\mathfrak{k})})$-equivariant.
Note also that $[\mathfrak{k}_i, \mathfrak{k}_j]=0$ for $i\neq j$, hence, the closure of $K=\exp(\mathfrak{k})$ in $G$ could be obtained
as a product of the closures
$\exp(\mathfrak{k}_i)$ in $G$.
Hence $\overline{\mathfrak{k}}$ is also adapted for (\ref{lieego1}). Therefore, $\overline{\mathfrak{k}}=\mathfrak{k}$ and $K$ is closed in $G$.
By Proposition \ref{nonfull5}, we get that (\ref{lieego1}) is geodesic orbit.
\end{proof}

\begin{prop}\label{nonfull7} Suppose that the inner product {\rm(\ref{lieego1})} generates a geodesic orbit left-invariant Riemannian metric on $G$,
$\mathfrak{k}_i= \mathfrak{k} \cap \mathfrak{p}_i$, $\mathfrak{m}_i$ is the $B$-orthogonal complement to
$\mathfrak{k}_i$ in $\mathfrak{p}_i$.
Then there is a maximal by inclusion adapted Lie subalgebra $\mathfrak{k}$ such that one of the following assertions hold:

{\rm 1)} There are no more than $1$ indices $i$ such that $\mathfrak{k}_i \neq \mathfrak{p}_i$. In this case
{\rm(\ref{lieego1})} generates a naturally reductive left-invariant Riemannian metric on $G$.

{\rm 2)} $\operatorname{rank}(\mathfrak{k})\geq 2$, and $[\mathfrak{m}_i,\mathfrak{m}_j]\subset \mathfrak{m}_i\oplus \mathfrak{m}_j$ for $i\neq j$.

{\rm 3)} There is only one non-zero $\mathfrak{k}_i= \mathfrak{k} \cap \mathfrak{p}_i$, hence, $\mathfrak{k}_i=\mathfrak{k}$,
moreover, $\operatorname{rank}(\mathfrak{k})= 1$ and ether $[\mathfrak{m}_i,\mathfrak{m}_j]\subset \mathfrak{m}_i$ or
$[\mathfrak{m}_i,\mathfrak{m}_j]\subset \mathfrak{m}_j$ for $i \neq j$.
\end{prop}

\begin{proof} Let us fix a maximal by inclusion adapted Lie subalgebra $\mathfrak{k}$ in $\mathfrak{g}$ such that for any $X \in \mathfrak{g}$ there is
$W\in \mathfrak{k}$ with property $[A(X),X+W]=0$, where $A$ is the metric endomorphism (such a subalgebra does exist by Proposition \ref{nonfull6}).

Note that if $\mathfrak{k}_i= \mathfrak{k} \cap \mathfrak{p}_i = \mathfrak{p}_i$ for all $i$, then the inner product (\ref{lieego1}) is bi-invariant and
it generates a normal homogeneous  (in particular, naturally reductive) metric on $G$. If
$\mathfrak{k}_i= \mathfrak{k} \cap \mathfrak{p}_i \neq \mathfrak{p}_i$ only for one $i$, then (\ref{lieego1}) generate a naturally reductive metrics on $G$ by
Proposition \ref{natredgr}.

Now we suppose that $\mathfrak{k}_i= \mathfrak{k} \cap \mathfrak{p}_i \neq \mathfrak{p}_i$ for some two indices,
i.~e. $\mathfrak{m}_i$ is non-zero for some two indices. Without loss of generality,
we may think that these are $i=1$ and $i=2$. Let us take some  $Y_i$, $i=1,2$, in $\mathfrak{m}_i$  and any non-trivial
$V\in \mathfrak{k}_j= \mathfrak{k} \cap \mathfrak{p}_j$ for some index $j$.
Note that $[\mathfrak{k},\mathfrak{m}_i]\subset \mathfrak{m}_i$ and $[\mathfrak{m}_1,\mathfrak{m}_2]\subset \oplus_{i=1}^s\mathfrak{m}_i$
(since $\mathfrak{m}_i$ are $\ad(\mathfrak{k})$-invariant and $B(\mathfrak{m}_1,\mathfrak{m}_2)=0$, then
$B([U_1,U_2],W)=-B(U_2,[U_1,W])=0$ for $U_i \in \mathfrak{m}_i$ and $W \in \mathfrak{k}$). It is easy to see that $A(V+Y_1+Y_2)=u_j X+u_1 Y_1+u_2 Y_2$.
We know that there is $W\in \mathfrak{k}$ such that $[A(V+Y_1+Y_2), V+Y_1+Y_2+W]=0$, which is equivalent to
$$
(u_j-u_1)[V,Y_1]-u_1[W,Y_1]+(u_j-u_2)[V,Y_2]-u_1[W,Y_2]+(u_1-u_2)[Y_1,Y_2]+u_j[V,W]=0,
$$
therefore, $[V,W]=0$ and
$$
(u_2-u_1)[Y_1,Y_2]=[(u_j-u_1)V-u_1W,Y_1]+[(u_j-u_2)V -u_1 W,Y_2]\in \mathfrak{m}_1 \oplus \mathfrak{m}_2.
$$
Note that $[\mathfrak{m}_i, \mathfrak{m}_j]\in \mathfrak{m}_i \oplus \mathfrak{m}_j$ for every different $i \neq j$ by the same reason.

Now, if $V$ and $W$ are lineal independent, then $\operatorname{rank}(\mathfrak{k})\geq 2$. Let us suppose that
$\operatorname{rank}(\mathfrak{k})=1$ (this obviously imply that there is only one non-zero $\mathfrak{k}_i$),
then $W=c V$ for some real number $c$,
and the above equality is simplified  up to
$$
(u_2-u_1)[Y_1,Y_2]=[(u_j-u_1(1+c))V,Y_1]+[(u_j-u_2(1+c))V,Y_2].
$$
If we consider the vector $2V$ instead $V$ in the above equalities, then there is real number $\tilde{c}$ such that $W=\tilde{c} (2V)$. Hence we get
(with the same $Y_1$ and $Y_2$)
$$
(u_2-u_1)[Y_1,Y_2]=[(u_j-u_1(1+\tilde{c}))2V,Y_1]+[(u_j-u_2(1+\tilde{c}))2V,Y_2].
$$
From the last two equality we get
$$
(u_j-u_1(1+2\tilde{c}-c))[V,Y_1]+(u_j-u_2(1+2\tilde{c}-c))[V,Y_2]=0.
$$
Since $u_1\neq u_2$ then $u_j-u_1(1+2\tilde{c}-c)$ and $u_j-u_2(1+2\tilde{c}-c)$ could not be zero simultaneously.
Suppose without loss of generality that $u_j-u_2(1+2\tilde{c}-c)\neq 0$, then $[V,Y_2]=0$ and $[Y_1,Y_2]\in \mathfrak{m}_1$.
The same arguments show that for every vectors $Z_1 \in \mathfrak{m}_1$ and $Z_2 \in \mathfrak{m}_2$,
we have either $[Z_1,Z_2]\in \mathfrak{m}_1$ or $[Z_1,Z_2]\in \mathfrak{m}_2$.
If $[Y_1,Y_2]\neq 0$ then for every vectors $Z_1 \in \mathfrak{m}_1$ and $Z_2 \in \mathfrak{m}_2$ sufficiently closed  to $Y_1$ and $Y_2$ respectively,
we get $[Z_1,Z_2]\in \mathfrak{m}_1$. Since we may constitute bases of $\mathfrak{m}_i$, $i=1,2$, from vectors of these kinds, we see that
$[\mathfrak{m}_1,\mathfrak{m}_2]\subset \mathfrak{m}_1$.
\end{proof}

\section{A left-invariant Einstein metric on the group $G_2$ that \\ is not geodesic orbit}

let us recall some information about root systems of a compact simple
Lie algebra $(\mathfrak{g},\langle \cdot,\cdot \rangle=-B)$ with the
Killing form $B$, which can be find in books \cite{Burb4, Hel}.

Let us fix a Cartan subalgebra  $\mathfrak{t}$ (that is maximal abelian subalgebra) of Lie algebra~$\mathfrak{g}$.
There is a set $\Delta$ (\textit{root system}) of (non-zero) real-valued linear form $\alpha \in \mathfrak{t}^{\ast}$ on the Cartan subalgebra
$\mathfrak{t}$, that are called \textit{roots}. Let us consider some positive root system $\Delta^+ \subset \Delta$. Recall that for any
$\alpha \in \Delta$ exactly one of the roots $\pm \alpha$ is positive (we denote it by $|\alpha|$).
The Lie algebra $\mathfrak{g}$ admits a direct
$\langle \cdot,\cdot \rangle$-orthogonal decomposition
\begin{equation}\label{rsd}
\mathfrak{g}=\mathfrak{t}\oplus \bigoplus_{\alpha \in \Delta^+} \mathfrak{v}_{\alpha}
\end{equation}
into vector subspaces, where
each subspace $\mathfrak{v}_{\alpha}$ is  2-dimensional and
$\ad(\mathfrak{t})$-invariant. Using the
restriction (of non-degenerate) inner product $\langle \cdot,\cdot \rangle$ to
$\mathfrak{t}$, we will naturally identify  $\alpha$ with some vector
in $\mathfrak{t}$. Note that $[\mathfrak{v}_{\alpha}, \mathfrak{v}_{\alpha}]$ is one-dimensional subalgebra in~$\mathfrak{t}$ spanned on
the root $\alpha$,
and $[\mathfrak{v}_{\alpha}, \mathfrak{v}_{\alpha}]\oplus \mathfrak{v}_{\alpha}$ is a Lie algebra isomorphic to $su(2)$.
The vector subspaces $\mathfrak{v}_{\alpha}$, $\alpha \in \Delta^+$, admit bases
$\{U_{\alpha},V_{\alpha}\}$, such that $\langle U_{\alpha},U_{\alpha}\rangle=\langle V_{\alpha},V_{\alpha} \rangle= 1$,
$\langle U_{\alpha},V_{\alpha}\rangle=0$
and
\begin{equation}\label{N}
[H,U_{\alpha}]=\langle \alpha,H \rangle V_{\alpha},\quad
[H,V_{\alpha}]=-\langle \alpha,H \rangle U_{\alpha}, \quad \forall H\in \mathfrak{t}, \quad
[U_{\alpha},V_{\alpha}]={\alpha}.
\end{equation}

Note also, that $[\mathfrak{v}_{\alpha},\mathfrak{v}_{\beta}]=\mathfrak{v}_{\alpha+\beta}+\mathfrak{v}_{|\alpha-\beta|}$,
assuming  $\mathfrak{v}_{\gamma}:=\{0\}$ for $\gamma \notin \Delta^+$.
\medskip

For a positive root system $\Delta^+$ the (closed) Weyl chamber is defined by the equality
\begin{equation}\label{carcham}
C=C(\Delta^+):=\{H\in\mathfrak{t}\,|\,\langle \alpha,H\rangle \geq 0\,\, \forall \alpha\in \Delta^+\}.
\end{equation}
Recall some important properties of the Weyl group $W=W(\mathfrak{t})$ of the Lie algebra $\mathfrak{g}$, that acts on the Cartan subalgebra
$\mathfrak{t}$.

(i) For every root $\alpha \in \Delta \subset \mathfrak{t}$ the Weyl
group $W$ contains the orthogonal reflection $\varphi_{\alpha}$ in
the plane $P_{\alpha}$, which is orthogonal to the root $\alpha$
with respect to $\langle \cdot,\cdot \rangle$. It is easy to see that $\varphi_{\alpha}(H)=H-2\frac{\langle H,\alpha\rangle}
{\langle \alpha,\alpha\rangle}\alpha$, $H\in \mathfrak{t}$.

(ii) Reflections from (i) generate $W$.

(iii) The root system $\Delta$ is invariant under the action of the Weyl group $W$.

(iv) $W$ acts irreducible on $\mathfrak{t}$ and simply transitively on the set of positive root systems.
For any $H\in \mathfrak{t}$, there is $w\in W$, such that $w(H)\in  C(\Delta^+)$.

(v)  For any $X\in \mathfrak{g}$, there is an inner automorphism $\psi$ of $\mathfrak{g}$ such that $\psi(X)\!\in\!\mathfrak{t}$.
For any $w\in W$, there is an inner automorphism $\eta$ of $\mathfrak{g}$, such that $\mathfrak{t}$ is stable under $\eta$, and the restriction of $\eta$
to $\mathfrak{t}$ coincides with $w$.

(vi) The Weyl group $W$ acts transitively on the set of positive roots of  fixed length.
\smallskip

\medskip

Let us describe all
subalgebras of maximal rank in $\mathfrak{g}$ up to a conjugation with respect to
$\Ad(g)$, $g\in G$, such that
$\Ad(g)(\mathfrak{t})=\mathfrak{t}$. Any such Lie subalgebra
$\mathfrak{h}$ is defined by a class of pairwise $W$-isomorphic
closed symmetric root subsystems $A$ of $\Delta$, not equal to
$\Delta$. By definition, $A \subset \Delta$ is \textit{closed}, if
$\alpha,\beta \in A$ and $\alpha \pm \beta \in \Delta$ imply $\alpha
\pm \beta \in A$, and \textit{symmetric}, if $-\alpha \in A$
together with $\alpha \in A$. Then
\begin{equation}
\label{h} \mathfrak{h}=\mathfrak{t}\oplus  \bigoplus_{\alpha \in A \cap \Delta^+} \mathfrak{v}_{\alpha}.
\end{equation}

\medskip

Now, let us give a description of the root system $\Delta_{G_2}$ of
the Lie algebra $g_2$. There are two simple roots $\alpha,\beta \in \Delta_{G_2}$
such that $\angle(\alpha,\beta)=\frac{5\pi}{6}$ and
$|\alpha|=\sqrt{3}|\beta|$, where $|X|=\sqrt{\langle X,X \rangle}$. Then
$$
\Delta_{G_2}=\{\pm\alpha, \pm\beta, \pm(\alpha+\beta), \pm(\alpha+2\beta),
\pm(\alpha+3\beta), \pm(2\alpha+3\beta)\}.
$$
Moreover, the above vectors with plus signs constitute a positive root system $\Delta_{G_2}^+$, whereas $\pm\alpha,
\pm(\alpha+3\beta), \pm(2\alpha+3\beta)$ are all long roots.
One can easily see that all non $W$-isomorphic closed symmetric root
subsystems of $\Delta_{G_2}$, not equal to $\Delta_{G_2},$  are
$\emptyset,$ $\{\pm\alpha\},$ $\{\pm\beta\},$
$\{\pm\beta,\pm(2\alpha+3\beta)\},$ $\{\pm\alpha,\pm(\alpha+3\beta),
\pm(2\alpha+3\beta)\}.$

The first three cases give us the following
(generalized) flag manifolds respectively: $G_2/T^2,$ $G_2/SU(2)SO(2),$ and
$G_2/A_{1,3}SO(2),$ where $A_{1,3}$ is a Lie group with Lie
subalgebra of the type $A_1$ of index 3, see \cite{On}.

The last two closed symmetric root subsystems are maximal, so they
correspond to maximal Lie subalgebras in $g_2,$ which are
respectively isomorphic to $su(2)\oplus su(2)$ and $su(3)$ with the
corresponding compact connected Lie subgroups $SO(4)$ and $SU(3)$
and homogeneous spaces $G_2/SO(4)$ and $G_2/SU(3)=S^6$, compare with
\cite{On}. Note that $G_2/SO(4)$ is irreducible symmetric space and $G_2/SU(3)$ is non-symmetric irreducible, see
\cite{Bes}.
Note also that any geodesic orbit
Riemannian homogeneous manifold $(G_2/H,\mu)$
of positive Euler characteristic is either $G_2$-normal or $SO(7)$-normal, see Proposition 23 in \cite{BerNik}.

\medskip

Let us define  special metrics of the type (\ref{lieego1}) for the Lie algebra $g_2$ with $s=5$.
Let us put
\begin{eqnarray*}
\mathfrak{p}_1&=&[ \mathfrak{v}_{\alpha+2\beta}, \mathfrak{v}_{\alpha+2\beta}]\subset \mathfrak{t},\\
\mathfrak{p}_2&=&[ \mathfrak{v}_{\alpha}, \mathfrak{v}_{\alpha}] \oplus  \mathfrak{v}_{\alpha},\\
\mathfrak{p}_3&=& \mathfrak{v}_{\alpha+\beta} \oplus  \mathfrak{v}_{\beta},\\
\mathfrak{p}_4&=& \mathfrak{v}_{\alpha+2\beta},\\
\mathfrak{p}_5&=& \mathfrak{v}_{2\alpha+3\beta} \oplus  \mathfrak{v}_{\alpha+3\beta}.
\end{eqnarray*}
Clear that
$\mathfrak{t}=[ \mathfrak{v}_{\alpha+2\beta}, \mathfrak{v}_{\alpha+2\beta}]\oplus[ \mathfrak{v}_{\alpha}, \mathfrak{v}_{\alpha}]$,
$g_2=\mathfrak{p}_1\oplus\mathfrak{p}_2\oplus\mathfrak{p}_3\oplus\mathfrak{p}_4\oplus\mathfrak{p}_5$, and, moreover,
all modules $\mathfrak{p}_i$, $i=1,\dots,5$, are pairwise $\langle \cdot, \cdot \rangle$-orthogonal.
It is easy to see that
$\mathfrak{p}_1\oplus\mathfrak{p}_2\oplus\mathfrak{p}_4=su(2)\oplus su(2)$ and
$\mathfrak{p}_1\oplus\mathfrak{p}_2\oplus\mathfrak{p}_5=su(3)$.  Moreover, $[\mathfrak{p}_3,\mathfrak{p}_5]\subset\mathfrak{p}_4$,
$[\mathfrak{p}_4,\mathfrak{p}_5]\subset\mathfrak{p}_3$, $[\mathfrak{p}_3,\mathfrak{p}_4]\subset\mathfrak{p}_3 \oplus \mathfrak{p}_5$
with non-zero $[\mathfrak{p}_3,\mathfrak{p}_5]$ and $[\mathfrak{p}_4,\mathfrak{p}_5]$, and $[\mathfrak{p}_3,\mathfrak{p}_4]\not\subset\mathfrak{p}_3$.

Now, let us consider the metrics

\begin{equation}\label{lieego2}
(\cdot, \cdot)=u_1 \langle \cdot, \cdot \rangle |_{\mathfrak{p}_1}+u_2 \langle \cdot, \cdot \rangle |_{\mathfrak{p}_2}+
u_3 \langle \cdot, \cdot \rangle |_{\mathfrak{p}_1}+u_4 \langle \cdot, \cdot \rangle |_{\mathfrak{p}_2}+
u_5 \langle \cdot, \cdot \rangle |_{\mathfrak{p}_5}
\end{equation}
with $u_i>0$, $i=1,\dots,5$.

\begin{prop}\label{nonfull8} Let us consider the group $G_2$ supplied with the left-invariant Riemannian metric $\rho$, generated with
the inner product {\rm(\ref{lieego2})} with pairwise distinct $u_i>0$, $i=1,\dots,5$. Then the Riemannian manifold $(G_2, \rho)$ is not geodesic orbit.
\end{prop}

\begin{proof} Our proof is based on Proposition~\ref{nonfull7}. Suppose that $(G_2, \rho)$ is a geodesic orbit Riemannian manifold,
then at least one of the three assertions 1), 2), or 3) in Proposition~\ref{nonfull7} is fulfilled.
Let us consider the corresponding Lie subalgebra $\mathfrak{k}\subset g_2=\mathfrak{g}$ as in Proposition~\ref{nonfull7} and
note that $\mathfrak{m}=\oplus_{i=1}^5 \mathfrak{m}_i$ is $\ad(\mathfrak{k})$-invariant (see the notation in Proposition~\ref{nonfull7}).

Let us suppose that 1) is fulfilled, i.~e. there is at most $1$ index $i$ such that $\mathfrak{k}_i:= \mathfrak{k} \cap \mathfrak{p}_i \neq \mathfrak{p}_i$.
This means that the Lie subalgebra $\mathfrak{k}$ should contain at least four of modules $\mathfrak{p}_i$.
Clear, that there is no such subalgebra in $g_2$.

Now, let us suppose that 2) is fulfilled, i.~e.
$\operatorname{rank}(\mathfrak{k})\geq 2$, and $[\mathfrak{m}_i,\mathfrak{m}_j]\subset \mathfrak{m}_i\oplus \mathfrak{m}_j$ for $i\neq j$.
Since $\operatorname{rank}(g_2)=2$, then $\operatorname{rank}(\mathfrak{k})= 2$ in our case. This mean that at most two subalgebras
$\mathfrak{k}_i=\mathfrak{k} \cap \mathfrak{p}_i$ are non-zero.
Since
$[\mathfrak{p}_3,\mathfrak{p}_5]\subset\mathfrak{p}_4$,
$[\mathfrak{p}_4,\mathfrak{p}_5]\subset\mathfrak{p}_3$, and $[\mathfrak{p}_3,\mathfrak{p}_4]\subset\mathfrak{p}_3 \oplus \mathfrak{p}_5$, we get that
$\mathfrak{k}_i=\mathfrak{k}\cap \mathfrak{p}_i \neq 0$ at least for two indices from the set $\{3,4,5\}$
(if $\mathfrak{k}_j=\mathfrak{k}_i=0$ for $i,j \in \{3,4,5\}$, then  the inclusion
$[\mathfrak{p}_i,\mathfrak{p}_j]=[\mathfrak{m}_i,\mathfrak{m}_j]\subset \mathfrak{m}_i\oplus \mathfrak{m}_j$ is not valid).
Therefore, $\mathfrak{p}_1 = \mathfrak{m}_1 \subset \mathfrak{m}$ (otherwise $\operatorname{rank}(\mathfrak{k})> 2$).
Now, we get contradiction with $[\mathfrak{k},\mathfrak{m}]\subset \mathfrak{m}$ since $[\mathfrak{p}_1, \mathfrak{p}_i]\subset \mathfrak{p}_i$ for all $i$
and $[\mathfrak{p}_1, \mathfrak{k}_i]\neq 0$ if $\mathfrak{k}_i$ is non-trivial, $i=3,4,5$ (the latter is true due to (\ref{N}) and the fact that the root
$\alpha+2\beta$ is not orthogonal to each of the roots
${\alpha+\beta}$, ${\beta}$, ${\alpha+2\beta}$, ${2\alpha+3\beta}$, and ${\alpha+3\beta}$).

Finally, let us suppose that 3) is fulfilled, i.~e. there is only one non-zero
$\mathfrak{k}_i= \mathfrak{k} \cap \mathfrak{p}_i=\mathfrak{k}$,
$\operatorname{rank}(\mathfrak{k})= 1$ and
ether $[\mathfrak{m}_i,\mathfrak{m}_j]\subset \mathfrak{m}_i$ or
$[\mathfrak{m}_i,\mathfrak{m}_j]\subset \mathfrak{m}_j$ for $i \neq j$. But this contradicts to the relations
$[\mathfrak{p}_3,\mathfrak{p}_5]\subset\mathfrak{p}_4$,
$[\mathfrak{p}_4,\mathfrak{p}_5]\subset\mathfrak{p}_3$, $[\mathfrak{p}_3,\mathfrak{p}_4]\subset\mathfrak{p}_3 \oplus \mathfrak{p}_5$
and to the fact that at least two of the modules
$\mathfrak{p}_3,\mathfrak{p}_4,\mathfrak{p}_5$ do not intersect with  $\mathfrak{k}$.
\end{proof}

\begin{theorem}\label{einstnongo}
There is a left-invariant Riemannian  metric $\rho$ on the compact simple Lie group $G_2$ such that $(G_2,\rho)$ is Einstein but
is not  a geodesic orbit Riemannian manifold.
\end{theorem}

\begin{proof} In the paper \cite{CS}, I.~Chrysikos and Y.~Sakane classified left-invariant Einstein metrics  of the type (\ref{lieego2}) on the Lie group $G_2$,
see Section 4.4 in \cite{CS}. There are exactly 3 such metrics, up to a homothety.
Two of them, generated with the parameters $(u_1,u_2,u_3,u_4,u_5)=(1,1,1,1,1)$ and $(u_1,u_2,u_3,u_4,u_5)=(1,1,11/9,11/9,1)$,
are naturally reductive with respect to suitable isometry groups. The most interesting is the third one with the following parameters:
$$
u_1 \approx 1.0851961, \ u_2 \approx 0.69929486,  \ u_3 \approx 0.93245951, \ u_4 \approx 1.0225069 , \ u_5=1.
$$
Note that $u_3$ is a root of a polynomial with integer coefficient of degree 39, moreover,
$u_1$,  $u_2$, and $u_4$ are given by polynomials of  degree 40  of $u_3$ with coefficients of rational numbers
(our notations are different from the notations of \cite{CS}). In \cite{CS}, it is proved that the corresponding Riemannian metric $\rho$ is not naturally reductive.
From our Proposition~\ref{nonfull8}, we see that the corresponding Riemannian manifold $(G_2, \rho)$ is not even geodesic orbit.
\end{proof}

\medskip
We propose the following two questions, closely related with the obtained results.

\begin{quest}\label{que1}
What is the minimal dimension of compact Lie groups, admitting  left-invariant
Einstein Riemannian metrics that are not geodesic orbit?
\end{quest}

\begin{quest}\label{que2}
Which simple compact Lie groups admit  left-invariant
Einstein Riemannian metrics that are not geodesic orbit?
\end{quest}

Recall that there is a one-parameter family of pairwise non-isometric 5-dimensional Einstein solvmanifolds that are not geodesic orbit
(see the discussion in Introduction).
Therefore, if we omit ``compact'' in Question  \ref{que1}, the answer is $5$.

Note also that  a compact Lie group $G$ with a left-invariant Einstein Rimannian metric and a non-discrete center should be a flat torus \cite{Bes},
hence naturally reductive and geodesic orbit. Therefore, we may restrict our attention on semisimple compact Lie groups.

According to Theorem 3 in \cite{Ozeki}, for any left-invariant Riemannian metric $\rho$ on a compact simply connected Lie group $G$,
there is a left-invariant Riemannian metric $\rho'$ isometric to $\rho$ and which full connected isometry group contains in the products
of left and right translations on $G$ (see also p.~23 in \cite{DZ}). This observation could be helpful in the study of Question \ref{que2}.

For instance, let us show that a left-invariant Einstein metrics $\rho$ on the group $SU(2)\times SU(2)$ is not geodesic orbit
if and only if its full connected isometry group is $SU(2)\times SU(2)$ (note that there is no example of  $\rho$ with this property).
Recall, that we have only partial results on the classification of left-invariant Einstein metrics on this group~\cite{NikRod}.
If~$\rho$ is one of these metrics,
we may assume without loss of generality, that its full connected isometry group is locally isomorphic to a group
$(SU(2)\times SU(2))\times K$,
where $K \subset SU(2)\times SU(2)$ means the subgroup of right translations. If $\dim(K)\geq 1$, then $K$ contains some $S^1$-subgroup. Then by Theorem 2
in \cite{NikRod}, $\rho$ is isometric either to the standard metric on $SU(2)\times SU(2)$
(the full isometry group is locally isomorphic to $SU(2)^4$ in this case), or
to the standard metric on $SU(2)\times SU(2)\times SU(2) / \diag(SU(2))$. In both cases $\rho$ is normal, hence, geodesic orbit.
On the other hand, if $\dim(K)=0$, then $SU(2)\times SU(2)$ is the full connected isometry group of $\rho$.
Let us show that $\rho$ could not be geodesic orbit in this case.
Indeed, if $\rho$ is geodesic orbit, then it is a bi-invariant metric on the group
$SU(2)\times SU(2)$ (see e.~g. Proposition 8 in \cite{AN}), hence its full isometry group is locally isomorphic to $SU(2)^4$.

Since $SU(2)$ admits only $SO(4)$-normal Einstein metrics and there are no semisimple compact Lie group of dimensions $5$ and $7$,
then one may study Question \ref{que2} only for dimensions $\geq 8$. The first interesting examples (in addition to the group $SU(2)\times SU(2)$) could be
the groups $SU(3)$ and $SU(2)\times SU(2)\times SU(2)$.

\bigskip

{\bf Acknowledgements.}
The author is indebted to Prof. Ioannis Chrysikos  and to Prof. Yusuke Sakane
for helpful discussions concerning this paper.

\vspace{10mm}

\bibliographystyle{amsunsrt}

\end{document}